\pdfoutput=1

\documentclass{amsart}
\usepackage[a4paper, left=30mm,right=30mm,top=35mm,bottom=35mm]{geometry}
\usepackage{amsmath}
\usepackage{amsthm}
\usepackage{amssymb}
\usepackage{graphicx}
\usepackage{color}
\usepackage{setspace}
\usepackage{float}
\usepackage{hyperref}
\hypersetup{
	colorlinks=true,
	linkcolor=blue,
}

\usepackage{amssymb}
\usepackage{amscd}
\usepackage{amsmath}
\usepackage{amssymb}
\usepackage{amsfonts}

\usepackage{amsmath}
\usepackage{amscd}
\usepackage{amsthm}

\newtheorem{thm}{Theorem}[section]
\newtheorem{coro}[thm]{Corollary}
\newtheorem{prop}[thm]{Proposition}
\newtheorem{lemma}[thm]{Lemma}
\newtheorem{conjecture}[thm]{Conjecture}

\theoremstyle{definition}
\newtheorem{definition}[thm]{Definition}
\newtheorem{remark}[thm]{Remark}
\theoremstyle{remark}

\DeclareMathOperator{\Cl}{Cl}
\DeclareMathOperator{\Det}{Det}

\DeclareMathOperator{\Gal}{Gal}

\DeclareMathOperator{\Hom}{Hom}

\newcommand{\QQ}{\mathbb{Q}}

\newcommand{\ZZ}{\mathbb{Z}}
\newcommand{\Z}{\mathbb{Z}}

\newcommand{\calM}{\mathcal{M}}

\def\bigcapp{\raise1ex\hbox{\rotatebox{180}{$\biguplus$}}}
\def\bigcappd{\raise1ex\hbox{\rotatebox{180}{$\displaystyle\biguplus$}}}

\begin{document}
\title[On the Galois-Gauss sums of weakly ramified characters]{On the Galois-Gauss sums\\
of weakly ramified characters}
\author{Y. Kuang}
\date{06 March 2023}
\email[Y. Kuang]{yu.kuang.3@gmail.com}
\address{Zhuhai, China}
	
\begin{abstract}
Bley, Burns and Hahn used relative algebraic $K$-theory methods to formulate a precise conjectural link between the (second Adams-operator twisted) Galois-Gauss sums of weakly ramified Artin characters and the square root of the inverse different of finite, odd degree, Galois extensions of number fields. We provide concrete new evidence for this conjecture in the setting of extensions of odd prime-power degree by using a refined version of a well-known result of Ullom.
\end{abstract}
	
\maketitle

\section{Introduction}
The subject of `classical' Galois module theory was developed by Fr\"ohlich, M. Taylor and others during the 1970's and 1980's (see \cite{F83}). 

At the heart of this theory was the amazingly close interplay between arithmetic properties of the Galois-Gauss sums that are associated to tamely ramified Artin characters and the structure of Hermitian-Galois modules that are attached to rings of algebraic integers in tamely ramified extensions of number fields. 

Much subsequent work has been dedicated to attempts to prove a natural extension of this theory in the setting of Artin characters that are not tamely ramified. 

Nevertheless, even today, the arithmetic properties of the Galois-Gauss sums of such characters are still, in general, very poorly understood.

However, for one special class of wildly ramified characters, important progress has been made. 

To discuss this, we fix an {\em odd degree} Galois extension of number fields $L/K$, with $G := {\rm Gal}(L/K)$, and we write $\mathcal{D}_{L/K}$ for the different of $L/K$. Then, by Hilbert's classical formula for the valuation of $\mathcal{D}_{L/K}$, there exists a unique $G$-stable fractional ideal $\mathcal{A}_{L/K}$ of $L$ such that
\[ (\mathcal{A}_{L/K})^2 =(\mathcal{D}_{L/K})^{-1}.  \] 

Further, if $L/K$ is `weakly ramified' in the sense of \cite{E} (i.e. the second ramification subgroup in $G$ of every place of $K$ is trivial, which, in particular, permits $L/K$ to be wildly ramified), then Erez has shown that $\mathcal{A}_{L/K}$ is a projective $G$-module and has suggested that the Hermitian-Galois structure associated to this module should provide an appropriate analogue of the ring of integers for such extensions.

This aspect of the theory of $\mathcal{A}_{L/K}$ was developed in a series of articles of Erez and others and culminated in the article \cite{ET} of Erez and Taylor which showed that for {\em tamely} ramified extensions $L/K$ of odd degree, the Hermitian-Galois structures that are attached to $\mathcal{A}_{L/K}$ are completely determined by the `twisted' Galois-Gauss sums that are obtained by combining the classical Galois-Gauss sum with the action on characters of a second Adams operator. 

It is, of course, natural to hope that such results could be extended to the setting of extensions that are weakly and wildly (rather than tamely) ramified. 

This has, however, proved to be a very difficult, and rather technical, problem and the most impressive results to date are contained in a series of articles of Vinatier (see, for example, \cite{V} and \cite{V2}), of Pickett and Vinatier \cite{PV}, of Pickett and Thomas \cite{PT} in which weakly (wildly) ramified extensions satisfying a variety of restrictive hypotheses are considered. 

Next we note that, taking inspiration from a very different direction (relating to leading term conjectures), Bley, Burns and Hahn \cite{BBH} (and see also the related PhD thesis \cite{hahn} of Hahn) have more recently introduced techniques of relative algebraic $K$-theory to this area of research in order to formulate a precise conjectural link between twisted Galois-Gauss sums of weakly ramified Artin characters and the square root of the inverse different. 

In this way, they have, in particular, been able to explain how much of the theory developed by Erez, by Erez and Taylor and by Vinatier could be refined in a natural way.

To be more precise we write ${\rm Cl}(\mathbb{Z}[G])$ for the reduced projective class group of $\mathbb{Z}[G]$ and $K_0(\mathbb{Z}[G],\mathbb{Q}^c[G])$ for the relative algebraic $K_0$-group of the ring inclusion $\mathbb{Z}[G]\subset \mathbb{Q}^c[G]$.

Then in \cite{BBH}, the authors describe a canonical pre-image $\mathfrak{a}_{L/K}$ of the stable-isomorphism class $[\mathcal{A}_{L/K}]$ of the (projective) module $\mathcal{A}_{L/K}$ in ${\rm Cl}(\mathbb{Z}[G])$ under the (surjective) connecting homomorphism of relative $K$-theory 
$\partial_G: K_0(\mathbb{Z}[G],\mathbb{Q}^c[G])\to {\rm Cl}(\mathbb{Z}[G])$. 

We recall that the kernel of the homomorphism $\partial_G$ is very large: for example, if $G$ is trivial, then ${\rm Cl}(\mathbb{Z}[G])$ vanishes, whilst the kernel of $\partial_G$ is isomorphic to the multiplicative group $\QQ^{c\times}/\{\pm 1\}$. For this reason, the study of $\mathfrak{a}_{L/K}$ is, at least in principle, likely to be much more difficult than that of $[\mathcal{A}_{L/K}]$.

Nevertheless, by using techniques of classical Galois module theory, Bley, Burns and Hahn were able to prove (in \cite[Th. 5.2]{BBH}) that $\mathfrak{a}_{L/K}$ belongs to the torsion subgroup $K_0(\mathbb{Z}[G],\mathbb{Q}[G])_{\rm tor}$ of the subgroup $K_0(\mathbb{Z}[G],\mathbb{Q}[G])$ of $K_0(\mathbb{Z}[G],\mathbb{Q}^c[G])$, that the elements $\mathfrak{a}_{L/K}$ have good functorial properties under change of extension $L/K$ and that $\mathfrak{a}_{L/K}$ simultaneously controls both the Hermitian-Galois structures and metrized structures that arise naturally from the (self-dual) module $\mathcal{A}_{L/K}$.

To study the element $\mathfrak{a}_{L/K}$ further, the authors also define a canonical variant of the classical `unramified characteristic' that has played a key role in the results of Fr\"ohlich, Taylor et al. The elementary definition of this `idelic twisted unramified characteristic' $\mathfrak{c}_{L/K}$ of $L/K$  makes it clear that it also belongs to $K_0(\mathbb{Z}[G],\mathbb{Q}[G])_{\rm tor}$ and enjoys the same functoriality properties under change of extension as does $\mathfrak{a}_{L/K}$. 

When taken together with extensive numerical computations, these facts motivated Bley, Burns and Hahn to formulate the following conjecture.

\begin{conjecture}[{\cite[Conjecture 10.7]{BBH}}]\label{bbh conj} If $L/K$ is any weakly ramified Galois extension of number fields of odd degree, then one has $\mathfrak{a}_{L/K} = \mathfrak{c}_{L/K}$.
\end{conjecture}

This concrete conjecture provides the motivation for much of the work that we undertake in this article, and is itself interesting for several reasons. 

Firstly, we recall (from the introduction to \cite{BBH}) that the approach underlying the formulation of Conjecture \ref{bbh conj} provides the first link between the theory of $\mathcal{A}_{L/K}$ and  the very general framework of
Tamagawa number conjectures that originates with Bloch and Kato in \cite{bk} and was subsequently refined to the relevant `equivariant' setting by Burns and Flach in \cite{BF} (this link, though important, will however play no role in the results that we present in this article).

More concretely, we note that, since the element $\mathfrak{c}_{L/K}$ is both very elementary in nature and easy to compute explicitly, the above conjecture suggests that the link between the twisted Galois-Gauss sums of weakly ramified Artin characters of $G$ and the invariants of $\mathcal{A}_{L/K}$ that are incorporated into the definition of $\mathfrak{a}_{L/K}$ should (or, at least, could!) be much finer than has previously been observed in the work of Erez, of Erez and Taylor or of Vinatier.  

We recall that the strongest theoretical evidence in support of Conjecture \ref{bbh conj} that is currently available is the following result (taken from \cite[Cor. 8.4]{BBH}). 

\begin{prop}[{Bley, Burns and Hahn}]\label{Th: BBH thm}
Let $L/K$ be a weakly ramified finite Galois extension of number fields of odd degree. Then the following claims are valid. 

\begin{itemize}
\item[(i)] The element $ \mathfrak{a}_{L/K} - \mathfrak{c}_{L/K}$ belongs to $K_0(\ZZ[G],\QQ[G])_{\rm tor}$. 
		
\item[(ii)] Conjecture~\ref{bbh conj} is valid provided that every place $v$ of $K$ that is wildly ramified in $L$ has the following three properties:
\begin{itemize}
\item[(a)] the decomposition subgroup in $G$ of any place of $L$ above $v$ is abelian;
\item[(b)] the inertia subgroup in $G$ of any place of $L$ above $v$ is cyclic;
\item[(c)] the completion of $K$ at $v$ is absolutely unramified.
\end{itemize}
\end{itemize}
\end{prop}

The proof of claim (i) of this result is reduced (by straightforward functoriality arguments, and a general result of Burns \cite{B} concerning $K_0(\ZZ[G],\QQ[G])_{\rm tor}$) to consideration of tamely ramified extensions. 

The proof of claim (ii), however,  depends heavily on detailed technical computations in certain families of wildly ramified extensions that are made by Pickett and Vinatier in \cite{PV}. For this reason, obtaining a full verification of Conjecture \ref{bbh conj} for classes of extensions that go beyond those in Proposition \ref{Th: BBH thm}(ii) seems to be a very difficult problem.

The main purpose of this paper is to provide further supporting evidence for Conjecture \ref{bbh conj} by proving the following result. 

\begin{thm}\label{Thm: p-group main intro} Fix an odd prime $p$ and a weakly ramified Galois extension of number fields $L/K$ of $p$-power degree, and set $G := \Gal(L/K)$. Suppose $p$ is unramified in $K$ and write $n(L/K)$ for the maximal order of a decomposition subgroup in $G$ of a wildly ramified place. Then, in $K_0(\ZZ[G],\QQ[G])$, one has 
\[ \frac{n(L/K)}{p} \cdot \mathfrak{a}_{L/K} = \frac{n(L/K)}{p} \cdot \mathfrak{c}_{L/K} = 0.\]
\end{thm} 

This result is of interest since there are many weakly ramified Galois extensions $L/K$ of odd prime-power degree in which the decomposition subgroup of each wildly ramified place is non-abelian (we recall, for example, the explicit extensions constructed by Vinatier in \cite{V-N}). 

Theorem \ref{Thm: p-group main intro} will be proved by combining (both parts of) Proposition \ref{Th: BBH thm} with a refined version, proved in Theorem \ref{Prop: odd degree tor group order}, of a well-known result of Ullom \cite{ullom0} concerning the structure of the kernel subgroup of ${\rm Cl}(\ZZ[G])$ in the case that $G$ has prime-power order. The latter result itself complements recent work of Bley and Wilson \cite{BW} concerning the explicit computation of relative algebraic $K$-groups and so is, we hope, of some independent interest. 

We next recall that, if $L/K$ is tamely ramified (and of odd degree), then in \cite[Th. 3]{E} Erez used the techniques of classical Galois module theory to prove that $[\mathcal{A}_{L/K}]$ vanishes and hence, by Jacobinski's Cancellation Theorem (which is valid since $G$ has odd order, see \cite[Th. (51.3) and Th. (51.24)]{CR2}), that $\mathcal{A}_{L/K}$ is a free $\ZZ[G]$-module. 

Motivated by this result, and several other factors (including extensive numerical computations), Vinatier was then led to make the  following conjecture (cf. \cite[Conj.]{V} and \cite[\S 1.2]{CV}). 

\begin{conjecture}[{Vinatier}] \label{C:sv}
If $L/K$ is a weakly ramified Galois extension of number fields of odd degree, then $\mathcal{A}_{L/K}$ is a free $\ZZ[G]$-module (or equivalently, the class $[\mathcal{A}_{L/K}]$ vanishes). 
\end{conjecture}

As evidence for this conjecture, Vinatier has proved explicit upper bounds on the order of the class $[\mathcal{A}_{L/K}]$ in the case that $K = \QQ$ and $G$ is of prime-power order (for details see the articles \cite{V} and \cite{V2}). 

If one projects the result of Theorem \ref{Thm: p-group main intro} into ${\rm Cl}(\ZZ[G])$, then one obtains a bound on the order of $[\mathcal{A}_{L/K}]$ that is weaker than the corresponding bounds established by Vinatier. The reason for this is that a key step in the approach of Vinatier cannot, as far as we can see, be repeated in the setting of relative algebraic $K$-groups (see Remark \ref{fin remark}). 

Nevertheless, we are able to show that the approach developed by Vinatier in \cite{V} can be used to obtain concrete new information about the relative element $\mathfrak{a}_{L/K}$ of Bley, Burns and Hahn. 

This result is proved as Theorem \ref{vinatier lift} and, whilst it does not (as far as we can currently see) provide more evidence in support of Conjecture~\ref{bbh conj} than is given by Theorem \ref{Thm: p-group main intro}, it does certainly provide further evidence for the belief that the link between $\mathcal{A}_{L/K}$ and Galois-Gauss sums should be closer than has so far been observed.   

\vskip 0.2truein \noindent{}{\bf Acknowledgements} 
I would like to thank David Burns for suggesting the problem and many inspiring discussions. I am very grateful to St\'ephane Vinatier for helpful, and encouraging, correspondence. Finally, I am very grateful to the anonymous referee for their careful review and valuable feedback, particularly, for having pointed out an error in an earlier version of the proof of Theorem \ref{Prop: odd degree tor group order} and for motivating me to obtain a sharper bound in the class group that is consistent with Vinatier's result.
\vskip 0.2truein

\section{Preliminaries}
\subsection{Notations}
For a unital ring $A$, we write $A^\times$ for the multiplicative group of invertible elements of $A$ and $\zeta(A)$ for the centre of $A$.

Fix a finite group $\Gamma$, and a Dedekind domain $R$ of characteristic zero, with field of fractions $F$, and $E$ is an extension field of $F$. We let $\mathcal{O}_{E}$ denote the integral closure of $R$ in $E$. For any $R[\Gamma]$-module $M$, we write $M_{E}:=E\otimes_{R} M$, and $M_{v}:=R_{v}\otimes_{R}M$, with $R_{v}$ the completion of $R$ at $v$, where $v$ is a (non-zero) prime ideal of $R$. 

Assume $E/F$ is a finite Galois extension of fields. We let ${\rm Gal}(E/F)$ denote the Galois group of $E/F$, and we write the action of ${\rm Gal}(E/F)$ on $E$ by $x \mapsto g(x)$ for $x \in E$ and $g\in {\rm Gal}(E/F)$. We fix a separable closure $F^c$ of $F$, and write $\Omega_{F}$ for the absolute Galois group ${\rm Gal}(F^c/F)$ of $F$. 
For convenience, we take $\mathbb{Q}^c$ to be the algebraic closure of $\mathbb{Q}$ in $\mathbb{C}$. 

For a finite group $\Gamma$, we write $\widehat{\Gamma}$ for the set of $\QQ^c$-valued irreducible characters of $\Gamma$ and $R_\Gamma$ for the additive group generated by $\widehat{\Gamma}$. In particular, for each $\chi$ in $\widehat{\Gamma}$, we obtain a primitive idempotent of the centre $\zeta(\QQ^c[\Gamma])$ of $\QQ^c[\Gamma]$ by setting 
$e_\chi =\frac{\chi(1)}{|\Gamma|} \sum_{g\in \Gamma}\chi(g^{-1})g$ and so each element of $\zeta(\QQ^c[\Gamma])$ (respectively $\zeta(\QQ^c[\Gamma])^\times$) can be written uniquely in the form
\begin{equation}\label{Eq: element in centre}
x= \sum_{\chi\in \widehat{\Gamma}}e_{\chi}\cdot x_{\chi} , \quad 
\text{ with $x_\chi\in \QQ^c$, respectively $x_\chi \in \QQ^{c\times}$, for all $\chi$}.
\end{equation}

\subsection{Algebraic $K$-theory}
In this section, we recall the definitions of relevant algebraic $K$-groups and the long exact sequences. Detailed descriptions of these definitions and results are given by Swan in \cite[Chap. 13 and 15]{swan} and also by Curtis and Reiner in \cite[Chap. 5 and 6]{CR2}.

We let $\mathcal{P}(E[\Gamma])$ denote the category of finitely generated projective $E[\Gamma]$-modules and ${\rm Aut}_{E[\Gamma]}(M)$ for the group of $E[\Gamma]$-module automorphisms of $M$. 

\subsubsection{Whitehead groups}
We write $K_1(E[\Gamma])$ for the Whitehead group of the ring $E[\Gamma]$. In particular, we recall that each element of $K_1(E[\Gamma])$ can be represented by a pair $[P, f]$ where $P\in \mathcal{P}(E[\Gamma])$ and $f\in \mathrm{Aut}_{E[\Gamma]}(P)$. (We refer the reader to \cite[\S38B]{CR2} for a detailed discussion of this group in terms of generators and relations.)

Suppose $E$ is either a number field or a $p$-adic field (for some prime $p$). Then, for some finite index set $I$, there is a Wedderburn decomposition  $E[\Gamma] = \prod_{i \in I} M_{n_i}(A_i)$ and each matrix ring $M_{n_i}(A_i)$ is a simple central algebra over the centre $\zeta(A_i)$ of $A_i$.  
In this case, one can compute in $K_1(E[\Gamma])$ by using the `reduced norm' map discussed by Curtis and Reiner in \cite[45A]{CR2} 
\begin{multline}\label{Eq: red norm Nrd def}
\mathrm{Nrd}_{E[\Gamma]}: K_{1}(E[\Gamma]) = \bigoplus_{i \in I}K_1(M_{n_i}(A_i))\cong  \bigoplus_{i \in I}K_1(A_i) \xrightarrow{({\rm Nrd}_{A_i})_i} \bigoplus_{i}\zeta(A_i)^\times = \zeta (E[\Gamma]) ^{\times}, 
\end{multline}
in which the isomorphism follows from the Morita equivalence (applied to each ring $M_{n_i}(A_i)$).

\subsubsection{The relative $K$-group}
We write $K_{0}(R[\Gamma], E[\Gamma])$ for the relative algebraic $K$-group associated to $R\subset E$ and recall that each element of this group is represented by a triple $[P, \phi , Q]$ where $P$, $Q \in \mathcal{P}(R[\Gamma])$ and $\phi \in {\rm Is}_{E[\Gamma]}(P_{E} , Q_{E})$, the set of $E[\Gamma]$-module isomorphisms from $P_E$ to $Q_E$. (For an explicit description of this group in terms of the generators and relations, see \cite[pp. 215]{swan}.)

Then there is a canonical decomposition
\begin{equation}\label{Eq: KT iso K0}
K_{0}(R[\Gamma], F[\Gamma])\xrightarrow{\sim} \bigoplus_{v}K_{0}(R_{v}[\Gamma], F_{v}[\Gamma]),
\end{equation}
where $v$ runs over all non-Archimedean places of $F$. This isomorphism is induced by the diagonal localisation homomorphism $(\pi_{\Gamma, v})_v$ (see, for example, discussion below \cite[(49.12)]{CR2}), where for each non-zero prime ideal $v$ of $R$, we write 
\begin{equation}\label{K_0 proj K_0_p}
\pi_{\Gamma, v}:  K_{0}(R[\Gamma], F[\Gamma]) \rightarrow K_{0}(R_{v}[\Gamma], F_{v}[\Gamma])
\end{equation}
for the homomorphism that sends the class of $[P, \phi, Q]$ to the class of $[P_{v}, F_{v}\otimes_{F}\phi, Q_{v}]$.

\subsubsection{The long exact sequence of relative $K$-theory}
Finally, we let ${\rm Cl}(R[\Gamma])$ denote the reduced projective class group of $R[\Gamma]$ (as discussed in \cite[\S49A]{CR2}) and recall from \cite[Th. 15.5]{swan} that there exists a commutative diagram 
\begin{equation} \label{Diagram: Kseq}
\begin{CD} 
K_1(R[\Gamma]) @> >> K_1(E[\Gamma]) @> \partial^1_{R,E,\Gamma} >> K_0(R[\Gamma],E[\Gamma]) @> \partial^0_{R,E,\Gamma} >> {\rm Cl}(R[\Gamma])\\
@\vert @A\iota_1 AA @A\iota_2 AA @\vert\\
K_1(R[\Gamma]) @> >> K_1(F[\Gamma]) @> \partial^1_{R,F,\Gamma}  >> K_0(R[\Gamma],F[\Gamma]) @> \partial^0_{R,F,\Gamma}  >> {\rm Cl}(R[\Gamma]),
\end{CD}
\end{equation}
in which the rows are the respective long exact sequences of relative $K$-theory. 

In the diagram, we have used the following notation (with the morphisms in the second row being completely analogous to those in the first): $\iota_{1}$ and $\iota_{2}$ are the natural scalar extension morphisms (these maps are injective and will usually be regarded as inclusions); 
the homomorphism $\partial^{1}_{R,E,\Gamma}$ sends each pair $[E[\Gamma]^{n}, \phi]$, with $\phi$ in $\mathrm{Aut}_{E[\Gamma]}(E[\Gamma]^{n})$, to the class of $[R[\Gamma]^{n}, f_\phi, R[\Gamma]^{n}]$, where $f_{\phi}$ is given by $(R[\Gamma]^n)_{E} \xrightarrow{\cong} E[\Gamma]^n \xrightarrow{\phi} E[\Gamma]^n \xrightarrow{\cong} (R[\Gamma]^n)_E$; 
for each pair $P$ and $Q$ in $\mathcal{P}(R[\Gamma])$ and each isomorphism of $E[\Gamma]$-modules $\phi: P_{E}\cong Q_{E}$, the homomorphism $\partial^{0}_{R,E,\Gamma}$ sends the class of $[P, \phi , Q]$ to the difference $[P]-[Q]$.

\subsubsection{The reduced norm map}
In this article, we shall always consider the cases that $\Gamma$ is a finite Galois group of odd order and $E$  is either a number field or a $p$-adic field (for some prime $p$). Therefore, by the Hasse-Schilling-Maass Norm Theorem (cf. \cite[Th. (7.48)]{CR1}) and \cite[Th. (7.45)]{CR1}, the reduced norm map ${\rm Nrd}_{E[\Gamma]}$ defined in \eqref{Eq: red norm Nrd def} is bijective and hence, we obtain a homomorphism by setting
\[\delta_{R, E, \Gamma}:=\partial^{1}_{R,E,\Gamma} \circ ({\rm Nrd}_{E[\Gamma]})^{-1}: \zeta(E[\Gamma])^{\times} \to K_0(R[\Gamma], E[\Gamma]). \]

Now for each subgroup $\Delta \subseteq \Gamma$, we write ${\rm i}^{\Gamma, *}_{\Delta, E} $ for the homomorphism of relative $K$-groups induced by the induction functor ${\rm i}^{\Gamma}_\Delta$ of $\mathcal{P}(R[\Delta]) \rightarrow \mathcal{P}(R[\Gamma])$ (via applying $R[\Gamma] \otimes_{R[\Delta]}$)
\begin{align}\label{Eq: K_0 ind def}
{\rm i}^{\Gamma, *}_{\Delta, E} : K_{0}(R[\Delta], E[\Delta])  \, & \rightarrow  K_{0}(R[\Gamma], E[\Gamma]),\\
[P, \phi, Q] \, & \mapsto  [{\rm i}^{\Gamma}_\Delta P, {\rm i}^{\Gamma}_\Delta \phi, {\rm i}^{\Gamma}_\Delta Q].\nonumber
\end{align}
We also let ${\rm res}^\Gamma_\Delta: R_{\Gamma} \rightarrow R_{\Delta}$ denote the restriction homomorphism and define a map $\mathrm{\tilde{i}}_{\Delta}^{\Gamma}: \zeta(E^c[\Delta])^{\times} \rightarrow \zeta(E^c[\Gamma])^{\times}$ by setting, for each $x \in \zeta(E^c[\Delta])^{\times}$ and $\chi \in \widehat{\Gamma}$ (in terms of \eqref{Eq: element in centre}),
\begin{equation}\label{Eq: center ind def}
\mathrm{\tilde{i}}_{\Delta}^{\Gamma} (x)_{\chi}= \prod_{\varphi \in \widehat{\Delta}} x_{\varphi}^{<\mathrm{res}^{\Gamma}_{\Delta} \chi , \varphi >_{\Delta}}. 
\end{equation} 
And we recall from \cite[pp. 581]{BB} that
\begin{equation}\label{Diagram: ind}
\mathrm{i}_{\Delta, E}^{\Gamma, *} \circ  \delta_{R, E, \Delta} =  \delta_{R, E, \Gamma}  \circ \mathrm{\tilde{i}}_{\Delta}^{\Gamma}.
\end{equation}

\subsection{Equivariant Galois-Gauss sums}\label{S: ggs}
Suppose $E/F$ is a finite Galois extension of $\ell$-adic fields and set $\Gamma:=\mathrm{Gal}(E/F)$. In terms of \eqref{Eq: element in centre}, we define the following `equivariant' elements in $\zeta(\mathbb{Q}^{c}[\Gamma])^\times$ as in \cite[\S4A]{BBH}.

\begin{definition}\label{Def: equiv global GGS}\
\begin{itemize}
\item[(i)] The `equivariant Galois-Gauss sum' of $E/F$ is the element $\tau_{E/F}\in \zeta(\mathbb{Q}^{c}[\Gamma])^\times$ defined by setting
\[\tau_{E/F}:=\sum_{\chi \in \widehat{\Gamma}} e_\chi \cdot \tau(F,\chi) .\]
where $\tau(F,\chi)$ is the (local) Galois Gauss sum discussed in \cite[Chap. I, \S 5]{F83}. 
\item[(ii)] For each $\phi \in \widehat{\Gamma}$, the `unramified characteristic' of $\phi$ is defined by setting (as in \cite[Chap. IV, \S1, (1.1)]{F83})
\[y(F, \phi) = \left\{ \begin{array}{ll}
1, & \quad \mbox{ if }\phi \mbox{ is ramified};\\
(-1)^{\phi(1)}\phi(\sigma), & \quad \mbox{ if }\phi \mbox{ is unramified},\end{array} \right. \]
where $\sigma$ is the Frobenius element in $\Gamma/\Gamma_{0}$ lifted to $\Gamma$ and $\Gamma_0$ is the inertia subgroup of $\Gamma$.
		
\item[(iii)] The `equivariant unramified characteristic' of $E/F$ is the element of $\zeta(\mathbb{Q}[\Gamma])^\times$ obtained by  setting
\[ y_{E/F}:=\sum_{\chi \in \widehat{\Gamma}} e_\chi \cdot y(F, \chi).  \]
\item[(iv)] The `modified equivariant Galois-Gauss sum' of $E/F$ is the element
\[ \tau'_{E/F}:=\tau_{E/F}\cdot y_{E/F}^{-1}\in \zeta (\mathbb{Q}^{c}[\Gamma])^\times.\]
\item[(v)] We set $\tau_F := \tau(\mathbb{Q}_\ell, \mathrm{ind}^{\mathbb{Q}_\ell}_{F} \mathbf{1}_{F})$ and define the `induced discriminant' of $E/F$ to be the element  $\tau^{\Gamma}_{F}$ of $\zeta(\mathbb{Q}^{c}[\Gamma])^{\times}$ with 
$$(\tau_{F}^{\Gamma})_{\chi}= \tau_F^{\chi(1)}$$
for all $\chi\in \widehat{\Gamma}$.
\end{itemize}	
\end{definition}

Last, for each natural number $k$ and each $\chi\in R_{\Gamma}$, we write $\psi_k$ for the $k$-th Adams operator on $R_\Gamma$ (cf. \cite[Prop. (12.8)]{CR1}). Then for each pair of integers $m$ and $n$, and each natural number $k$, we define an endomorphism $m + n\cdot \psi_{k,\ast}$ of $\zeta (\mathbb{Q}^{c}[\Gamma])^\times$ (in terms of \eqref{Eq: element in centre}) such that for each element $x$ of $\zeta (\mathbb{Q}^{c}[\Gamma])^\times$, the element $(m+n\cdot \psi_{k,*})(x)$ is uniquely specified by $((m+n\cdot \psi_{k,*})(x))_{\chi}:= (x_{\chi})^{m}\cdot (x_{\psi_{k}(\chi)})^{n}$ for each irreducible character $\chi$ in $\widehat{\Gamma}$.

\section{The torsion subgroup of $K_0(\ZZ[\Gamma],\QQ[\Gamma])$}\label{new refined kernel results}
Fix a finite group $\Gamma$ and let $R$ denote either $\ZZ$ or $\ZZ_\ell$ for a prime $\ell$, and $Q$ denote the corresponding fields $\QQ$ or $\QQ_\ell$. In this section we study the torsion subgroup
\[ {\rm DT}(R[\Gamma]) := K_{0}(R[\Gamma], Q[\Gamma])_{\rm tor}\]
of $K_{0}(R[\Gamma], Q[\Gamma])$. 

We recall that these groups have already been studied in the article \cite{BW} of Bley and Wilson that contains both general results and numerical algorithms.

To start, we first note that the isomorphism (\ref{Eq: KT iso K0}) restricts to give a canonical direct sum decomposition 
\begin{equation}\label{direct sum decomp DT}{\rm DT}(\ZZ[\Gamma]) \cong \bigoplus_{\ell} {\rm DT}(\ZZ_\ell[\Gamma]) \end{equation}
over all primes $\ell$. 

Next, for each prime $\ell$, we fix a maximal order $\mathcal{M}_\ell$ in $\QQ_\ell[\Gamma]$ that contains $\ZZ_\ell[\Gamma]$ and we note that $\calM_\ell$ is hereditary (by \cite[Th. (26.12)]{CR1}). Then the results of \cite[Th. 2.2(iv) and 2.4(i)]{BW} imply that 
\begin{equation}\label{tor}
{\rm DT}(\mathbb{Z}_{\ell}[\Gamma]) \cong  \frac{\mathrm{Nrd}_{\mathbb{Q}_{\ell}[\Gamma]}(\mathcal{M}_{\ell}^{\times})}{\mathrm{Nrd}_{\mathbb{Q}_{\ell}[\Gamma]}(\mathbb{Z}_{\ell}[\Gamma]^{\times})}.
\end{equation}
In the sequel, unless there is danger of confusion, we shall abbreviate $\mathrm{Nrd}_{\mathbb{Q}_{\ell}[\Gamma]}$ to $\mathrm{Nrd}$. This isomorphism then combines with the direct sum decomposition (\ref{direct sum decomp DT}) to give an isomorphism  
\begin{equation}\label{useful decomp}{\rm DT}(\mathbb{Z}[\Gamma])\cong \prod_{\ell} \frac{\mathrm{Nrd}(\mathcal{M}_{\ell}^{\times})}{\mathrm{Nrd}(\mathbb{Z}_{\ell}[\Gamma]^{\times})} = \prod_{\ell\in \Sigma(\Gamma)} \frac{\mathrm{Nrd}(\mathcal{M}_{\ell}^{\times})}{\mathrm{Nrd}(\mathbb{Z}_{\ell}[\Gamma]^{\times})}. \end{equation}
where $\Sigma(\Gamma)$ denotes the finite set of primes that divide the order of $\Gamma$ (or equivalently, at which $\mathcal{M}_{\ell}\neq \mathbb{Z}_{\ell}[\Gamma]$). This isomorphism implies that ${\rm DT}(\mathbb{Z}[\Gamma])$ is finite (see \cite[Cor. 2.5]{BW}). 

In the next result we prove finer results for ${\rm DT}(\mathbb{Z}[\Gamma])$ in the case that $\Gamma$ has prime power order. 
Claim (i) of this result is a version for relative $K$-groups of a well-known bound on the exponent of the kernel group $D(\ZZ[\Gamma])$ that is due to Ullom \cite{ullom0} (see also \cite[Th. (50.19)]{CR2}).

\begin{thm}\label{Prop: odd degree tor group order} 
Let $p$ be an odd prime and $\Gamma$ be a non-trivial group of order $p^n$ for a natural number $n$. Then, the following claims are valid.
\begin{itemize}
\item[(i)] The exponent of ${\rm DT}(\mathbb{Z}[\Gamma])$ is divisible by $p-1$ and divides $(1-1/p)|\Gamma|$.
\item[(ii)] Write $\Xi(\Gamma)$ for the set of normal subgroups of $\Gamma$ that are given by the kernels of the irreducible characters of $\Gamma$. Then, the exponent of the kernel of the diagonal projection map
\[ {\rm DT}(\mathbb{Z}[\Gamma]) \to \bigoplus_{\Delta \in \Xi(\Gamma)}{\rm DT}(\mathbb{Z}[\Gamma/\Delta])\]
divides $|\Gamma|/p$.
\item[(iii)] Claims (i) and (ii) remain valid if one replaces $\ZZ$ by $\ZZ_p$.
\end{itemize}
\end{thm}

\begin{proof} 
Fix a normal subgroup $\Delta$ of $\Gamma$ and set $e_\Delta:=|\Delta|^{-1} \sum_{g \in \Delta} g$. For each prime $\ell$ we write $\mathcal{M}_{\ell, \Gamma/\Delta} := e_\Delta\mathcal{M}_\ell$. Then, since the order $\mathcal{M}_\ell$ is maximal, the central idempotent $e_\Delta$ belongs to $\mathcal{M}_\ell$, and so $\mathcal{M}_\ell$ decomposes as a direct sum $(1-e_\Delta)\mathcal{M}_\ell \oplus \mathcal{M}_{\ell, \Gamma/\Delta}$. This decomposition implies that $\mathcal{M}_{\ell, \Gamma/\Delta}$ is a maximal order in $\QQ_\ell[\Gamma/\Delta]$, and also that the  natural map $\mathcal{M}_\ell^\times \to \mathcal{M}^\times_{\ell, \Gamma/\Delta}$ is surjective. By the isomorphism (\ref{useful decomp}), we can therefore deduce that the natural projection map 
\begin{equation}\label{canon surj} {\rm DT}(\mathbb{Z}[\Gamma]) \to {\rm DT}(\mathbb{Z}[\Gamma/\Delta])\end{equation}
is surjective. 
	
We now assume that $|\Gamma| = p^n$. In this case, the set $\Sigma(\Gamma)$ is equal to $\{p\}$ and so (\ref{useful decomp}) implies that the natural map ${\rm DT}(\mathbb{Z}[\Gamma]) \to {\rm DT}(\mathbb{Z}_p[\Gamma])$ is bijective. In particular, if we can prove claims (i) and (ii), then claim (iii) is clearly also true. 
	
To do this, we first choose $\Delta$ to have index $p$ (this can be done since, in this case, any subgroup of $\Gamma$ of index $p$ is normal), then ${\rm DT}(\mathbb{Z}[\Gamma/\Delta])$ is cyclic of order $p-1$ (see, for example, \cite[Cor. 8.2]{BW}), and so the surjection (\ref{canon surj}) implies that the exponent of ${\rm DT}(\mathbb{Z}[\Gamma])$ is divisible by $p-1$, as stated in claim (i). 
	
Next, to prove that the exponent of ${\rm DT}(\mathbb{Z}[\Gamma]) = {\rm DT}(\mathbb{Z}_p[\Gamma])$ divides $(p-1)p^{n-1}$, we recall that, by Schilling's Theorem (cf. \cite[Th. (74.17)]{CR2}), for a suitable natural number $t$, the Wedderburn decomposition of $\QQ_p[\Gamma]$ has the form 
\begin{equation}\label{wedd decomp} \mathbb{Q}_p[\Gamma] \cong \QQ_p\times \prod_{i =1}^{i=t}{\rm M}_{n_{i}}(F_{i}),\end{equation}
where each $n_i$ is a natural number and each $F_i$ is a $p$-adic cyclotomic field of the form $\mathbb{Q}_p(\omega_{i})$ for a suitable  non-trivial $p$-power root of unity $\omega_i$. In the sequel we set $F_0:=\QQ_p$, and for each $i$, we write $\mathcal{O}_i$ for the valuation ring of $F_i$ and $\mathfrak{p}_i$ for the maximal ideal of $\mathcal{O}_i$. 

In this way, by \cite[Th. 2.2(iv)]{BW} we have that
\[  {\rm Nrd}\bigl(\mathcal{M}^\times_p\bigr) = \prod_{i =0}^{i=t}\mathcal{O}^\times_{i}, \]
where the right hand side is the unit group of the integral closure of $\Z_p$ in the centre $\prod_{i=0}^{i=t}F_i$ of $\QQ_p[\Gamma]$. Therefore, each element $x$ of ${\rm Nrd}\bigl(\mathcal{M}^\times_p\bigr)$ can be written in the form 
\begin{equation}
\label{dt rep} x = \bigl(u_0,u_1,\cdots, u_t\bigr),
\end{equation}
with $u_i \in \mathcal{O}_i^\times$ for each $0\le i\le t$. Moreover, since each extension $F_i/\QQ_p$ is totally ramified, the natural map $\ZZ_p/(p) \to \mathcal{O}_i/\mathfrak{p}_i$ is bijective, and so one has $u_i^{p-1} \equiv 1 \,({\rm mod}\, \mathfrak{p}_i)$. 
	
It follows that $x^{p-1}$ is equal to a tuple $\bigl(v_0,v_1,\cdots ,v_t\bigr)$ in which each element $v_i$ belongs to $\mathcal{O}^\times_i$ and is such that $v_i \equiv 1 \,$ modulo $\,\mathfrak{p}_i$. In view of the isomorphism (\ref{tor}), in order to prove claim (i), it is therefore enough to show that the $p^{n-1}$-st power of any such tuple belongs to $\mathrm{Nrd}(\mathbb{Z}_p[\Gamma]^{\times})$. This is exactly what is proved by the argument of Ullom that is discussed on pp. 257-259 of \cite{CR2} (in which our tuple corresponds to the element $\alpha$), and so this concludes the proof of claim (i). 
	
To prove claim (ii), we note that the components in the Wedderburn decomposition (\ref{wedd decomp}) are in bijective correspondence with a set of representatives for the orbits of the natural action of $\Gal(\QQ_p^c/\QQ_p)$ on the set $\widehat{\Gamma}_p$ of irreducible $\QQ_p^c$-valued characters of $\Gamma$. More precisely, we fix a subset $\{\psi_j: 0\le j \le t\}$ of $\widehat{\Gamma}_p$ such that $\psi_0$ is the trivial character and set $F_0: = \QQ_p$, $\mathcal{O}_0 = \ZZ_p$ and $n_0 = 1$ for $j=0$. For each character $\psi_j$, we let $T_{j,*}$ denote the $\QQ_p$-linear map $\QQ_p[\Gamma] \to {\rm M}_{n_j}(F_j)$ that is induced by a choice of a representation $T_j: \Gamma \to {\rm GL}_{n_j}(F_j)$ of $\psi_j$. In this way, for each $j$, the projection map  
\[ \QQ_p[\Gamma] \to M_{n_j}(F_j) \]
that is induced by (\ref{wedd decomp}) coincides with $T_{j,*}$. 
We fix an embedding $\QQ^c \to \QQ_p^c$ and we shall identify $\widehat{\Gamma}$ with $\widehat{\Gamma}_p$ via the choice of the embedding.
Then, since $\Gamma$ is nilpotent of odd order, the results of Roquette in \cite{roquette} (or see the more general results of Cliff, Ritter and Weiss in \cite{crw}) imply that one can choose a $T_j$ (of $\psi_j$) such that it takes values in ${\rm GL}_{n_j}(\mathcal{O}_j)$ and this is what we shall do. 
	
Set $\Delta_j := \ker(\psi_j) \in \Xi(\Gamma)$ and $\Gamma_j := \Gamma/\Delta_j$, we note that $T_{j,\ast}$ factors through the projection map $\pi_j: \QQ_p[\Gamma] \to \QQ_p[\Gamma_j]$. We also note that $T_{0,\ast}$ is the natural augmentation map $\varepsilon_\Gamma: \QQ_p[\Gamma] \to \QQ_p$. 
	
Now, we fix an element $y$ in $\mathcal{M}_p^\times$. Then, for each $j$, the element $u_j$ that occurs in the decomposition (\ref{dt rep}) of $x={\rm Nrd}(y)$ belongs to $\mathcal{O}_j^\times$ and can be computed as ${\rm det}(T_{j,\ast}(y))$. 
	
To prove claim (ii), we claim that it is enough to show that if the image $\overline{x}$ of $x$ in ${\rm DT}(\ZZ_p[\Gamma])$ belongs to the kernel of the projection map $\pi_{i,*}: {\rm DT}(\ZZ_p[\Gamma]) \to {\rm DT}(\ZZ_p[\Gamma_i])$ for any given $i$ with $1\le i\le t$, then one has 
\begin{equation}\label{wanted cong} {\rm det}(T_{i,\ast}(y)) \equiv (T_{0,\ast}(y))^{\psi_i(1)}\,\,\,\, \pmod{\mathfrak{p}_i}.\end{equation}
To see that the above result implies claim (ii), we assume that these congruences are valid for each $1\le i\le t$. Set $v_i := u_0^{-\psi_i(1)}u_i$, we note that such element satisfies $v_i \equiv 1 \, $ modulo $\,\,\mathfrak{p}_i$ for all $i$. In particular, since $u_0$ belongs to $\ZZ_p^\times \subseteq \ZZ_p[\Gamma]^\times$, these congruences and the argument of claim (i) combine to imply that the element 
\begin{align*} x^{p^{n-1}} =&\, (u_0,u_1,\cdots u_t)^{p^{n-1}} \\
=&\, (u_0,u_0^{\psi_1(1)},\cdots , u_0^{\psi_t(1)})^{p^{n-1}}\times ( 1, v_1, \cdots , v_t)^{p^{n-1}}\\
=&\, {\rm Nrd}(u_0)^{p^{n-1}} \times ( 1, v_1, \cdots , v_t)^{p^{n-1}}\end{align*}
belongs to $\ZZ_p[\Gamma]^\times$, and hence that $\overline{x}^{p^{n-1}} = 0$, as required. 
	
To prove the stated congruence (\ref{wanted cong}), we assume that $\overline{x}$ belongs to the kernel of $\pi_{i,*}$. Then, the isomorphism (\ref{tor}) implies that there exists an element $z_i = \sum_{\gamma \in \Gamma_i}c_{i,\gamma}\gamma \in \ZZ_p[\Gamma_i]^\times$ such that 
\begin{equation}\label{aug point}
\pi_{i, *}  ({\rm Nrd}_{\QQ_p[\Gamma]}(y)) = {\rm Nrd}_{\QQ_p[\Gamma_i]}(z_i) \end{equation}
and hence,
\begin{equation}\label{useful cons} {\rm det}(T_{i,\ast}(y)) = {\rm det}(T_{i,\ast}(z_i)) = {\rm det}\bigl( \sum_{\gamma \in \Gamma_i}c_{i,\gamma} \cdot T_i(\gamma)\bigr).\end{equation}
	
Now, we note that (by the choice of $T_i$ as above) each matrix $T_i(\gamma)$ belongs to ${\rm GL}_{n_i}(\mathcal{O}_i)$. Upon setting $m_i := |\Gamma_i|$, the matrix $T_i(\gamma)^{m_i} = T_i(\gamma^{m_i})$ is equal to the $n_i\times n_i$ identity matrix $I_{n_i}$. Hence, since $m_i$ belongs to $\mathfrak{p}_i$, the binomial theorem implies that there are congruences modulo ${\rm M}_{n_i}(\mathfrak{p}_i)$ of the form  
\begin{align*}\bigl( \sum_{\gamma \in \Gamma_i}c_{i,\gamma} \cdot T_i(\gamma)\bigr)^{m_i} \equiv &\,\sum_{\gamma \in \Gamma_i}c^{m_i}_{i,\gamma} \cdot I_{n_i}\\
\equiv &\,\bigl( \sum_{\gamma \in \Gamma_i}c^{m_i}_{i,\gamma}\bigr) \cdot I_{n_i} \\
\equiv &\, \bigl( \sum_{\gamma \in \Gamma_i}c_{i,\gamma}\bigr)^{m_i}\cdot I_{n_i}\\
\equiv &\, (T_{0,\ast}(y))^{m_i}\cdot I_{n_i} \quad \quad \pmod{ {\rm M}_{n_i}(\mathfrak{p}_i)},
\end{align*}
where the last congruence is true since $T_{0, \ast}$ is the natural augmentation map and so, in terms of the expression \eqref{aug point}, we have that $T_{0,\ast}(y)$ is equal to $\sum_{\gamma \in \Gamma_i}c_{i,\gamma}$.  
	
Since $n_i = \psi_i(1)$, these congruences and the equality (\ref{useful cons}) combine to imply that 
\begin{align*} 
{\rm det}(T_{i,\ast}(y))^{m_i} =&\, {\rm det}\bigl(\bigl( \sum_{\gamma \in \Gamma_i}c_{i,\gamma} \cdot T_i(\gamma)\bigr)^{m_i}\bigr) \\
\equiv &\, {\rm det}\bigl((T_{0,\ast}(y))^{m_i}\cdot I_{n_i}\bigr)  \quad \quad \pmod{\mathfrak{p}_i } \\
\equiv &\, \bigl((T_{0,\ast}(y))^{\psi_i(1)}\bigr)^{m_i} \quad \quad \pmod{\mathfrak{p}_i }.
\end{align*}
Finally, since ${\rm det}(T_{i,\ast}(x))$ and $(T_{0,\ast}(x))^{\psi_i(1)}$ both belong to $\mathcal{O}_i^\times$ (and the order of $\bigl(\mathcal{O}_i/\mathfrak{p}_i\bigr)^\times$ is prime to $m_i$), the above congruence then implies the required congruence (\ref{wanted cong}). This completes the proof of the claim. \end{proof}

\begin{remark}\label{abelian case remark}\
\begin{itemize} 
\item[(i)] The result of Theorem \ref{Prop: odd degree tor group order}(ii) is of interest only if $\Xi(\Gamma)$ does not contain the trivial subgroup of $\Gamma$. This condition is equivalent to requiring that $\Gamma$ has no faithful irreducible characters and such groups are completely classified by Gaschutz in \cite{Gaschutz}. 
\item[(ii)] If $\Gamma$ is abelian, then the set $\Xi(\Gamma)$ is equal to the set of subgroups $\Delta$ of $\Gamma$ with the property that the quotient group $\Gamma/\Delta$ is cyclic. 
\end{itemize} 
\end{remark}

\section{Results for extensions of odd prime-power degree}
In this section, we try to provide some new evidence for Conjecture~\ref{bbh conj} in the case of extensions of odd prime-power degree.

\subsection{The canonical relative elements of Bley, Burns and Hahn}
Firstly, we recall the key definition of the canonical relative elements of Bley, Burns and Hahn from \cite{BBH}.

\subsubsection{The local element}\label{S: local elt}
In this section, we fix a finite odd degree Galois extension $E/F$ of local fields of residue characteristic $\ell$ and set $\Gamma := {\rm Gal}(E/F)$. We follow Bley, Burns and Hahn \cite[\S 7A]{BBH} in defining the local relative elements.

We write $\Sigma(E)$ for the set of $\QQ_\ell$-linear field embeddings $E\rightarrow \mathbb{Q}_{\ell}^{c}$. We then define a (free) $\ZZ_\ell[\Gamma]$-module by setting 
\[ H_{E}:=\prod_{\Sigma(E)}\mathbb{Z}_{\ell},\]
upon which $\Gamma$ acts via its usual pre-composition action on $\Sigma(E)$, and we consider the isomorphism of $\QQ_\ell^c[\Gamma]$-modules 
\[ \kappa_{E}:\mathbb{Q}^{c}_{\ell}\otimes_{\mathbb{Q}_{\ell}}E \rightarrow \prod_{\Sigma(E)} \mathbb{Q}_{\ell}^{c} = \mathbb{Q}^{c}_{\ell}\otimes_{\ZZ_{\ell}}H_E\]
that sends $x \otimes l$, for every $x \in \mathbb{Q}_{\ell}^{c}$ and $l$ in $E$ to $(\sigma(l)x)_{\sigma \in \Sigma(E)}$. 

Suppose $E/F$ is weakly ramified. Then, by the result \cite[Th. 1]{E} of Erez, $\mathcal{A}_{E/F}$ is a full projective $\ZZ_\ell[\Gamma]$-sublattice of $E$ and so, following \cite[\S 7A]{BBH}, we obtain a well-defined element of $K_0(\ZZ_\ell[\Gamma],\QQ_\ell^c[\Gamma])$ by setting 
\[ \Delta(\mathcal{A}_{E/F}) := [\mathcal{A}_{E/F},\kappa_{E}, H_{E}].\] 

For the next definition, we use the following notation: for each prime $\ell$ and embedding of fields $j^c_{\ell}:\mathbb{Q}^{c}\rightarrow \mathbb{Q}^{c}_{\ell}$, we also write $j^c_\ell$ for the induced homomorphism of rings $\QQ^c[G] \to \QQ^c_\ell[G]$ and consider the homomorphism of abelian groups 
\begin{align}\label{jlc def}
 j^c_{\ell, *}: K_{0}(\mathbb{Z}[\Gamma], \mathbb{Q}^c[\Gamma]) &\, \to K_{0}(\mathbb{Z}_\ell[\Gamma],\mathbb{Q}_\ell^c[\Gamma]), \\ 
[P, \phi, Q] &\, \mapsto  [P_\ell, \mathbb{Q}_\ell^c\otimes_{\mathbb{Q}^c, j^c_\ell}\phi, Q_\ell].   \nonumber
\end{align}

\begin{definition}\label{local a def} We set 
\[ T^{(2)}_{E/F}:= \tau_{F}^{\Gamma}\cdot (\psi_{2,*}-1)(\tau_{E/F}^{\prime}) \in \zeta(\QQ^c[\Gamma])^\times,\]
where the elements $\tau^{\Gamma}_{F}$ and $\tau_{E/F}^{\prime}$ are as in \S\ref{S: ggs}. 

We then define an element of $K_{0}(\mathbb{Z}_{\ell}[\Gamma], \mathbb{Q}_{\ell}^{c}[\Gamma])$ by setting 
\[\mathfrak{a}_{E/F}:= \Delta(\mathcal{A}_{E/F}) - \delta_{\ZZ_\ell,\QQ_\ell^c,\Gamma}(j^c_{\ell}(T^{(2)}_{E/F})) - U_{E/F},  \]
where $U_{E/F}$ is the canonical `unramified' element of $K_{0}(\mathbb{Z}_{\ell}[\Gamma], \mathbb{Q}_{\ell}^{c}[\Gamma])$ that is defined by Breuning in \cite{B04}.
\end{definition}

In particular, we note (by \cite[Prop. 7.1]{BBH}) that $\mathfrak{a}_{E/F}$ is independent of the choice of embedding $j_\ell^c$ and belongs to $K_{0}(\mathbb{Z}_{\ell}[\Gamma], \mathbb{Q}_{\ell}[\Gamma])$.

\subsubsection{The global element}
Let $L/K$ be a weakly ramified finite Galois extension of number fields of odd degree with Galois group $G:={\rm Gal}(L/K)$. In \cite[\S 2A3 and \S 5]{BBH}, Bley, Burns and Hahn define a global canonical relative elements of $K_0(\ZZ[G],\QQ^c[G])$ 
\[\mathfrak{a}_{L/K}:= \Delta(\mathcal{A}_{L/K}) - \delta_{\ZZ,\QQ^c,G}(T^{(2)}_{L/K}) ,  \]
where $\Delta(\mathcal{A}_{L/K})$ and $T^{(2)}_{L/K}$ are the global analogues of the elements defined above.

Fix a place $w$ of $L$ above a place $v$ of $K$, we write $G_w$ for the decomposition subgroup of $w$ in $G$, and for each $\chi \in R_G$, we let $\chi_w$ denote the restriction of $\chi$ to $G_w$ (therefore regard $\chi_w$ as characters of ${\rm Gal}(L_w/K_v)$ via identifying $G_w={\rm Gal}(L_w/K_v)$). In this article, we will often use the decomposition property of the relative element (taken from Theorem 7.6 of loc.cit.) that 
\begin{equation}\label{Eq: frak a decomp}
\mathfrak{a}_{L/K} = \sum_{\ell}\sum_{v|\ell} {\rm i}_{G_{w}, \mathbb{Q}_{\ell}}^{G, *} (\mathfrak{a}_{L_{w}/K_{v}}),
\end{equation}
where the first sum runs over all rational primes $\ell$ and the second over all places $v$ of $K$ of residue characteristic $\ell$.

\begin{remark}\label{tame local remark} 
Since (by \cite[Prop. 5.5]{BBH}) $\mathfrak{a}_{L/K}$ belongs to $K_0(\ZZ[G],\QQ[G])$, the sum on the right hand side of the formula in \eqref{Eq: frak a decomp} can only have finitely many non-zero $\ell$-primary components. In fact, Bley, Burns and Hahn \cite[Th. 8.1]{BBH} have shown that $\mathfrak{a}_{L_w/K_v}$ vanishes if the extension $L_w/K_v$ is tamely ramified (which is true for all $v$ that do not divide the order of $G$). \end{remark}

Next we recall a variant of the classical `unramified characteristic' that is introduced in loc. cit.

\begin{definition}\label{Def: loc c} \
\begin{itemize}
\item[(i)] Let $E/F$ be a finite odd Galois extension of $\ell$-adic fields and set $\Gamma := \Gal(E/F)$. Then, the `twisted unramified characteristic' of $E/F$ is the element of $K_{0}(\mathbb{Z}_{\ell}[\Gamma], \mathbb{Q}_{\ell}[\Gamma])$ that is obtained by setting
\[\mathfrak{c}_{E/F}:= \delta_{\ZZ_\ell,\QQ_\ell,\Gamma}((1 - \psi_{2, *})(y_{E/F}))  , \]
where $y_{E/F}$ is the equivariant unramified characteristic defined in \S\ref{S: ggs}. In particular, we recall from \cite[Rem. 7.5]{BBH} that $\mathfrak{c}_{E/F}=0$ if $E/F$ is tamely ramified.
\item[(ii)] Let $L/K$ be a finite odd Galois extension of number fields and set $G:={\rm Gal}(L/K)$. Then, the `idelic twisted unramified characteristic' of $L/K$ is defined to be the element 
\[\mathfrak{c}_{L/K} = \sum_{\ell}\sum_{v|\ell} {\rm i}_{G_{w}, \mathbb{Q}_{\ell}}^{G, *} (\mathfrak{c}_{L_{w}/K_{v}}) \in K_{0}(\mathbb{Z}[G], \mathbb{Q}[G]). \]
\end{itemize}
\end{definition}

\begin{remark}\label{functorial remark} 
Conjecture \ref{bbh conj} is compatible with the change of extension functors. To see this, we recall from \cite[Th. 6.1 and Rem. 8.9]{BBH} that, for any subgroup $J$ of $G$, the following equalities are valid.
\begin{itemize}
\item[(i)] $\rho^{G,*}_{J}(\mathfrak{a}_{L/K}) = \mathfrak{a}_{L/L^J}$ and $\rho^{G,*}_{J}(\mathfrak{c}_{L/K}) = \mathfrak{c}_{L/L^J}$, where $\rho^{G,*}_J$ denotes the restriction map  $K_{0}(\mathbb{Z}[G], \mathbb{Q}[G]) \rightarrow K_{0}(\mathbb{Z}[J], \mathbb{Q}[J])$.
\item[(ii)] If $J$ is normal to $G$, with $Q := G/J$, then $q^{G,*}_{Q}(\mathfrak{a}_{L/K}) = \mathfrak{a}_{L^J/K}$ and 
$q^{G,*}_{Q}(\mathfrak{c}_{L/K}) = \mathfrak{c}_{L^J/K}$, where $q^{G}_{Q}$ denotes the coinflation map $K_{0}(\mathbb{Z}[G], \mathbb{Q}[G])\rightarrow K_{0}(\mathbb{Z}[Q], \mathbb{Q}[Q])$.
\end{itemize}
\end{remark}

\subsection{Proof of Theorem~\ref{Thm: p-group main intro}}
The motivation for this section is the fact that there are many weakly ramified extensions $L/K$ of $p$-power degree in which the decomposition subgroup of every $p$-adic place is non-abelian, so Proposition \ref{Th: BBH thm}(ii) does not apply directly. (For example, in \cite[Th. 1.4]{V-N}, Vinatier has constructed an infinite family of weakly ramified Galois extensions of $K = \QQ$ of degree $27$ in which the decomposition subgroup of each $3$-adic place of $L$ is non-abelian.)

However, by combining Proposition \ref{Th: BBH thm} with the result of Theorem \ref{Prop: odd degree tor group order}, we are able to obtain an upper bound on the order of the element $\mathfrak{a}_{L/K}$ and can thereby provide evidence in support of Conjecture~\ref{bbh conj} for a more general class of extensions. 

Taking account of the decomposition of $\mathfrak{a}_{L/K}$ given in \eqref{Eq: frak a decomp}, the result of Bley, Burns and Hahn recalled in Remark \ref{tame local remark} and the definition of $\mathfrak{c}_{L/K}$ as an explicit sum of local terms (see Definition~\ref{Def: loc c}(ii)), it is enough for us to show that the stated equalities hold for each ($p$-adic) place $v$ of $K$ that ramifies wildly in $L$.

\begin{prop}\label{p power thm1 pre prop}
Fix an odd prime $p$. Suppose that $L/K$ is a weakly ramified Galois extension of number fields of $p$-power degree and that $p$ is unramified in $K$. If $v$ is any $p$-adic place of $K$ that is wildly ramified in $L$, we write $G_w:=\Gal(L_w/K_v)$ for some place $w$ of $L$ above $v$ and $n_w := |G_w|$. Then, in $K_0(\ZZ_p[G_w],\QQ_p[G_w])$, one has
\begin{equation}\label{reduced equalities} 
\frac{n_w}{p}\cdot \mathfrak{a}_{L_w/K_v} = \frac{n_w}{p}\cdot \mathfrak{c}_{L_w/K_v} = 0.
\end{equation} 
\end{prop}

In the sequel, we fix an embedding $j_p^c:\mathbb{Q}^c \rightarrow \mathbb{Q}_p^c$. For a finite group $\Gamma$, we shall use $j_p^c$ to identify $\widehat{\Gamma}$ with $\widehat{\Gamma}_p$, the set of $\QQ^c_p$-valued irreducible characters of $\Gamma$ (via the isomorphism $\chi \mapsto \chi^{j}$ where $\chi^{j}(g):= j_p^c(\chi(g))$ for all $\chi \in \widehat{\Gamma}$ and $g \in \Gamma$). We also write $j_p^c$ for the induced embedding $\zeta(\QQ^c[\Gamma])^\times \rightarrow \zeta(\QQ^c_p[\Gamma])^\times$.
Let $\mathcal{O}_p^t$ denote the valuation ring of the maximal tamely ramified extension of $\mathbb{Q}_p$, we consider the associated homomorphism of relative $K$-groups,
\begin{align}\label{Eq: Taylor map def0}
j^{t}_{p, *} : K_{0}(\mathbb{Z}_{p}[\Gamma], \mathbb{Q}^{c}_{p}[\Gamma]) &\, \to K_{0}(\mathcal{O}_{p}^{t}[\Gamma], \mathbb{Q}^{c}_{p}[\Gamma]),\\
[P, \phi , Q] &\, \mapsto [\mathcal{O}_{p}^{t}\otimes_{\mathbb{Z}_{p}}P , \phi,  \mathcal{O}_{p}^{t}\otimes_{\mathbb{Z}_{p}} Q].\nonumber
\end{align}
To prove the equalities in \eqref{reduced equalities}, we need to show some preliminary results. 

For the next three results, we fix a weakly ramified Galois extension of $p$-adic fields $E/F$ of arbitrary odd degree, set $\Gamma := \Gal(E/F)$ and write $\Delta$ for the inertia subgroup of $\Gamma$. We also write $\Gamma_1$ (resp. $\Gamma_2$) for the $1$-st (resp. $2$-nd) ramification subgroup (in the lower numbering) of $\Gamma$. The following result is a consequence of the hypothesis that $E/F$ is weakly ramified (i.e. $\Gamma_2$ is trivial).

\begin{lemma}\label{Prop: local abelian inertia p-group}\
\begin{itemize}
\item[(i)] The group $\Gamma_1$ is abelian and of exponent dividing $p$. 
\item[(ii)] The inertia subgroup $\Delta$ is abelian if and only if it is either a $p$-group (and hence, equal to $\Gamma_1$) or cyclic and of order prime-to-$p$.  
\end{itemize}
\end{lemma}
\begin{proof} 
At the outset, we note that the group $\Gamma_2$ is a normal subgroup of $\Gamma_1$, and the quotient group $\Gamma_1/\Gamma_2$ is abelian and of exponent dividing $p$ (by \cite[Chap. IV, \S2, Cor. 3 to Prop. 7]{S2}). Then, claim (i) follows directly from the hypothesis that $E/F$ is weakly ramified. 

Turning to claim (ii). If $\Delta$ has order prime-to-$p$, then it is cyclic (since, in this case, $E/F$ is tamely ramified). If $\Delta$ is a $p$-group, then it is equal to $\Gamma_1$ and hence abelian (by claim (i)). 
	
Now, to show the converse, 
we assume $\Delta$ is abelian and set $e_0 := |\Delta/\Gamma_1|$. It is enough for us to consider the case that $e_0>1$. In this case, upon applying the result of \cite[Chap. IV, \S2, Cor. 2 to Prop. 9]{S2} to the abelian extension $E/E^{\Delta}$, one can deduce that $\Gamma_1 = \Gamma_2$ is trivial. It follows that $E/F$ is tamely ramified, and hence that $\Delta$ is cyclic. This completes the proof of claim (ii).  
\end{proof}

The second of the required equalities (\ref{reduced equalities}) follows directly from the next result. 

\begin{lemma}\label{easy part} 
If $E/F$ is as above, then in $K_0(\ZZ_p[\Gamma],\QQ_p[\Gamma])$, one has $|\Gamma/\Delta|\cdot \mathfrak{c}_{E/F} =0$. 
\end{lemma}  

\begin{proof} Fix an element $\sigma$ of $\Gamma$ that projects to give the Frobenius automorphism in $\Gamma/\Delta$. 
	
By \cite[Rem. 7.5]{BBH}, the element $\mathfrak{c}_{E/F}$ is equal to the image under $\delta_{\ZZ_p,\QQ_p,\Gamma}$ of the element
\begin{equation}\label{Eq: frak c exp}
(1 - \psi_{2, *})(y_{E/F}) = (1-e_\Delta) + \sigma^{-1} e_\Delta \in \zeta(\QQ_p[\Gamma])^\times.
\end{equation}
Then, the claimed result is true since, 
\[ \left((1-e_\Delta) + \sigma^{-1} e_\Delta\right)^{|\Gamma/\Delta|} = (1-e_\Delta) + (\sigma^{-1} e_\Delta)^{|\Gamma/\Delta|} = (1-e_\Delta) + e_\Delta = 1.\]
\end{proof}

In the next result, we write $B$ for the maximal unramified extension of $F$ in $E$ (so that $\Delta = \Gal(E/B)$). We also recall from \cite[Chap. I, \S4 and Chap. III, \S3, (3.1)]{F83} that the \textit{resolvent} element and the (local) \emph{norm resolvent} is defined by setting, for each element $a$ in $E$ generating a normal basis of $E/F$, and each character $\chi$ of representation $T_\chi : \Gamma \rightarrow {\rm GL}_{n}(\QQ_p^c)$,
\begin{equation*}\label{Eq: resolvent def}
(a|\chi):=\mathrm{det}(\sum_{g\in \Gamma} g(a) T_{\chi}(g^{-1})),  \quad \quad \mathcal{N}_{F/\mathbb{Q}_\ell}(a|\chi)= \prod_{\omega}(a|\chi^{\omega^{-1}})^{\omega}, 
\end{equation*}
where for the second product $\omega$ runs through a transversal of $\Omega_{F}$ in $\Omega_{\mathbb{Q}_{p}}$.

\begin{prop}\label{key prop ?} 
In $K_0(\ZZ_p[\Gamma],\QQ_p[\Gamma])$, one has $|\Gamma/\Delta|\cdot \mathfrak{a}_{E/F} = {\rm i}^{\Gamma, *}_{\Delta, \mathbb{Q}_p}(\mathfrak{a}_{E/B}).
$\end{prop}
\begin{proof} By Taylor's Fixed Point Theorem (cf. \cite[Chap. 8, \S1]{M84}), one knows that the restriction to $K_0(\ZZ_p[\Gamma], \QQ_p[\Gamma])$ of the homomorphism $j^{t}_{p, *}$ (defined in \eqref{Eq: Taylor map def0}) is injective, so it is enough to prove that the images in $K_{0}(\mathcal{O}_{p}^{t}[\Gamma], \mathbb{Q}^{c}_p[\Gamma])$ of the respective elements 
$|\Gamma/\Delta|\cdot \mathfrak{a}_{E/F}$ and ${\rm i}^{\Gamma, *}_{\Delta, \mathbb{Q}_p}(\mathfrak{a}_{E/B})$ coincide. 
	
For $L \in \{B,F\}$, with $\Gamma_L := \Gal(E/L)$, we set $\delta_{\Gamma_L,p} := \delta_{\ZZ_p,\QQ_p^c,\Gamma_L}$ and 
\[ \mathfrak{a}'_{E/L}:= \Delta(\mathcal{A}_{E/L}) - \delta_{\Gamma_L,p}(j^c_{p}(T^{(2)}_{E/L}))
\in K_{0}(\ZZ_p[\Gamma_L], \mathbb{Q}^{c}_p[\Gamma_L]).\]
Then, the result \cite[Prop. 4.4]{B_phd} of Breuning that $U_{E/L} \in \ker(j_{p, *}^{t})$ implies that the images in $K_{0}(\mathcal{O}^t_p[\Gamma_L], \mathbb{Q}^{c}_p[\Gamma_L])$ of the elements $\mathfrak{a}_{E/L}$ and $\mathfrak{a}'_{E/L}$ coincide.
It is therefore enough for us to prove that the images in $K_{0}(\mathcal{O}_{p}^{t}[\Gamma], \mathbb{Q}^{c}_p[\Gamma])$ of the elements $|\Gamma/\Delta|\cdot \mathfrak{a}'_{E/F}$ and ${\rm i}^{\Gamma, *}_{\Delta, \mathbb{Q}_p}(\mathfrak{a}'_{E/B})$ coincide.

To do this, we fix a $\mathbb{Z}_p$-basis $\{a_\sigma \}_{\Sigma(L)}$ of $\mathcal{O}_L$ and set 
\begin{equation}\label{delta def}
\delta_{L}:=\det(\tau(a_\sigma) )_{ \tau, \sigma \in \Sigma(L)} \in \QQ_p^c
\end{equation}
and $u_L := \delta_L/j_p^c(\tau_L)$. We also  fix an element $a_L$ of $E$ such that $\mathcal{A}_{E/L}= \mathcal{O}_{L}[\Gamma_L]\cdot a_L$.
Then, by the explicit formula \cite[Lem. 4.16]{B_phd} of Breuning, one has 
\begin{align*} \mathfrak{a}'_{E/L} =&\, \delta_{\Gamma_L,p}\left( 
(j^c_{p}(T^{(2)}_{E/L}))^{-1}\sum_{\chi\in \widehat{\Gamma}_L} (\delta_{L}^{\chi(1)} \cdot \mathcal{N}_{L/\mathbb{Q}_{p}}(a_L|\chi)) e_{\chi}\right)\\
=&\, \delta_{\Gamma_L,p}\left( \sum_{\chi\in \widehat{\Gamma}_L}\left((\delta_{L}/ j_p^c(\tau_L))^{\chi(1)} \cdot \frac{\mathcal{N}_{L/\mathbb{Q}_{p}}(a_L|\chi)}{j^c_{p}(\tau(L,\psi_2(\chi)-\chi) y(L,\chi - \psi_2(\chi)))}\right) e_{\chi}\right)\\
=&\, \delta_{\Gamma_L,p}\bigl( x_L\bigr) + \delta_{\Gamma_L,p}\left(\sum_{\chi\in \widehat{\Gamma}_L}\frac{\mathcal{N}_{L/\mathbb{Q}_{p}}(a_L|\chi)}{j^c_{p}(\tau(L,\psi_2(\chi)-\chi)\cdot y(L, \chi - \psi_2(\chi)))} e_{\chi}\right), \end{align*}
with $x_L := \sum_{\chi\in \widehat{\Gamma}_L} u_L^{\chi(1)}e_\chi = {\rm Nrd}_{\QQ_p^c[\Gamma_L]}(u_L)$. 
	
In the sequel, we omit the occurrence of $j^c_{p}$ in each of our notations.
	
Now, since $u_L$ is a unit of $\mathcal{O}_p^t$ (by \cite[Lem. 4.29]{B_phd}), the image in $K_{0}(\mathcal{O}_{p}^{t}[\Gamma_L], \mathbb{Q}^{c}_p[\Gamma_L])$ of $\delta_{\Gamma_L,p}(x_L)$ vanishes. 
	
In addition, one has $y(B, \chi - \psi_2(\chi)) = 1$ for every $\chi \in \widehat \Delta$ since the extension $E/B$ is totally ramified (and so there is no non-trivial unramified character). One can also compute (directly from Definition~\ref{Def: equiv global GGS}(ii)) that $y(F, \chi -\psi_2(\chi))^{|\Gamma/\Delta|} = 1$ for every $\chi \in \widehat \Gamma$.

In view of \eqref{Diagram: ind}, one can write 
\begin{align}\label{ind formula}
{\rm i}^{\Gamma, *}_{\Delta, \mathbb{Q}_p} \circ   \delta_{\Delta ,p}
\left(\sum_{\phi\in \widehat{\Delta}}\frac{\mathcal{N}_{B/\mathbb{Q}_{p}}(a_B|\phi)}{j^c_{p}(\tau(B,\psi_2(\phi)-\phi))} e_{\phi}\right)
=\, & \delta_{\Gamma, p} \circ \tilde{\rm i}^{\Gamma}_{\Delta} 
\left(\sum_{\phi\in \widehat{\Delta}}\frac{\mathcal{N}_{B/\mathbb{Q}_{p}}(a_B|\phi)}{j^c_{p}(\tau(B,\psi_2(\phi)-\phi))} e_{\phi}\right)  \nonumber \\ 
=\, & \delta_{\Gamma, p} 
\left(\sum_{\chi\in \widehat{\Gamma}}
\frac{\mathcal{N}_{B/\mathbb{Q}_{p}}(a_B|{\rm res}^\Gamma_\Delta\chi)}{\tau(B,{\rm res}^\Gamma_\Delta(\psi_2(\chi)-\chi))} e_{\chi}
\right),
\end{align}
where the second equality follows from the definition of the induction map \eqref{Eq: center ind def} and the fact that $\tau(B, \psi_2({\rm res}^\Gamma_\Delta \chi)-{\rm res}^\Gamma_\Delta\chi) = \tau(B, {\rm res}^\Gamma_\Delta(\psi_2(\chi)-\chi)) $ (cf. \cite[Prop.-Def. 3.5]{E}).

In order to show that the images in $K_{0}(\mathcal{O}_{p}^{t}[\Gamma], \mathbb{Q}^{c}_p[\Gamma])$ of the elements $|\Gamma/\Delta|\cdot \mathfrak{a}'_{E/F}$ and ${\rm i}^{\Gamma, *}_{\Delta, \mathbb{Q}_p}(\mathfrak{a}'_{E/B})$ coincide, it is therefore enough for us to show that the elements 
\begin{equation}\label{p group ind}
\sum_{\chi\in \widehat{\Gamma}}\frac{(\mathcal{N}_{F/\mathbb{Q}_{p}}(a_F|\chi))^{[B:F]}}{\tau(F,\psi_2(\chi)-\chi)^{[B:F]}} e_{\chi} \quad \text{ and } \quad \sum_{\chi\in \widehat{\Gamma}}\frac{\mathcal{N}_{B/\mathbb{Q}_{p}}(a_B|{\rm res}^\Gamma_\Delta\chi)}{\tau(B,{\rm res}^\Gamma_\Delta(\psi_2(\chi)-\chi))} e_{\chi}
\end{equation}
of $\zeta(\QQ_p^c[\Gamma])^\times$ differ by an element of ${\rm Nrd}_{\QQ^c_p[\Gamma]}(\ZZ_p[\Gamma]^\times)$. 
	
To prove this, we note that,
by the argument of \cite[Th. 25 and the following Remark]{F83} (which is valid even for extensions that are wildly ramified, as observed by Erez in \cite[\S6]{E}), for each $\chi$ in $\widehat\Gamma$, one has 
\[ \tau(F,\psi_2(\chi)-\chi)^{[B:F]} = \tau(B, {\rm res}^\Gamma_\Delta(\psi_2(\chi)-\chi)), \]
since $\psi_2(\chi)-\chi$ has degree zero, and for a suitable choice of generating element $a_B$, one also has 
\[ (\mathcal{N}_{F/\mathbb{Q}_{p}}(a_F|\chi))^{[B:F]} = \mathcal{N}_{B/\mathbb{Q}_{p}}(a_B|{\rm res}^{\Gamma}_{\Delta} \chi)\cdot {\rm Det}_\chi(\lambda),\]
where the element $\lambda$ of $\ZZ_p[\Gamma]^\times$ is independent of $\chi$. This completes the proof of the claimed result.
\end{proof} 

We are now ready to prove Proposition~\ref{p power thm1 pre prop}.

For $v$ and $w$ as in (\ref{reduced equalities}), we now write $I_w$ for the inertia subgroup of $G_w$ and $B_w$ for the fixed field of $I_w$ in $L_w$. 

Then, in view of the last result (for $E/F = L_w/K_v$) and the obvious equality $|G_w|/p = |I_w|/p \times |G_w/I_w|$, 
to prove the first equality in (\ref{reduced equalities}), it is therefore enough for us to show that the element $ \mathfrak{a}_{L_w/B_w}$ of $K_0(\ZZ_p[I_w],\QQ_p[I_w])$ is annihilated by $|I_w|/p$. 

To do this, we write $L_0$ for the fixed field of $I_w$ in $L$ and $w_0$ for the place of $L_0$ obtained by restricting $w$. Then, since $p$ is unramified in $K$, the extension $L_{0,w_0}/\QQ_p$ is unramified and so $L_0$ cannot contain a non-trivial $p$-th root of unity. By the result \cite[Cor. 2, p. 156]{neuk solvable} of Neukirch, we can therefore fix a finite Galois extension $L'$ of $L_0$ with both of the following properties:
\begin{itemize}
\item[(i)] $L'$ has a unique place $w'$ above $w_0$ and the completion $L'_{w'}/L_{0,w_0}$ is isomorphic to $L_w/B_w$;
\item[(ii)] if $v'$ is any place of $L_0$ which divides $|I_w|$, and $v'\not= w_0$, then $v'$ is totally split in $L'/L_0$.
\end{itemize}

These conditions imply that the extension $L'/L_0$ is weakly ramified and that the group $G' := \Gal(L'/L_0)$ identifies with $\Gal(L'_{w'}/L_{0,w_0}) \cong \Gal(L_w/B_w) = I_w.$ In particular, since $w_0$ is the only place of $L_0$ that ramifies wildly in $L'$, in this case the decomposition result in \eqref{Eq: frak a decomp} combines with the vanishing of $\mathfrak{a}_{L'/L_0}$ for tamely ramified extensions of $\ell$-adic fields (see Remark \ref{tame local remark} with $L/K = L'/L_0$) to imply that  
\[  \mathfrak{a}_{L'/L_0} = \mathfrak{a}_{L'_{w'}/L_{0,w_0}} = \mathfrak{a}_{L_w/B_w}.\]

It is therefore sufficient for us to show that $\mathfrak{a}_{L'/L_0}$ is annihilated by $|G'|/p$. To do this we note that Proposition \ref{Th: BBH thm}(i) implies that $\mathfrak{a}_{L'/L_0}$ belongs to ${\rm DT}(\ZZ[G'])$. 

In addition, Lemma \ref{Prop: local abelian inertia p-group}(ii) implies that the group $G' (\cong I_w)$ is abelian and so the set $\Xi(G')$ in Theorem \ref{Prop: odd degree tor group order}(ii) is equal to the set of normal subgroups $H$ of $G'$ with the property that $G'/H$ is cyclic (see Remark \ref{abelian case remark}(ii)). 

Now, we fix $H$ in $\Xi(G')$ and set $\widetilde{L'} := (L')^H$ and $\widetilde{G} := G'/H \cong \Gal(\widetilde{L'}/L_0)$. Then, Remark \ref{functorial remark}(ii) implies that $q^{G',*}_{\tilde{G}}(\mathfrak{a}_{L'/L_0}) = \mathfrak{a}_{\widetilde{L'}/L_0}$. In addition, since $\widetilde{G}$ is cyclic of $p$-power order and  $L_{0, w_0 }$ is an unramified extension of $\QQ_p$, Proposition \ref{Th: BBH thm}(ii) implies that $\mathfrak{a}_{\widetilde{L'}/L_0} = \mathfrak{c}_{\widetilde{L'}/L_0}$. Finally, we note that, since  $w_0$ is totally ramified in $\widetilde{L'}$, the expression \eqref{Eq: frak c exp} implies that $\mathfrak{c}_{\widetilde{L'}/L_0}$ vanishes. 

This argument shows that $\mathfrak{a}_{L'/L_0}$ belongs to the kernel of the diagonal map that occurs in Theorem \ref{Prop: odd degree tor group order}(ii) (with $\Gamma = G'$) and so the latter result implies that $\mathfrak{a}_{L'/L_0}$ is annihilated by $|G'|/p$, as required. 

This finishes the proof of Proposition~\ref{p power thm1 pre prop} and hence, Theorem \ref{Thm: p-group main intro}.

\begin{remark}\label{remark above}
Alternatively, one can complete the argument of Proposition~\ref{p power thm1 pre prop} by proving an analogous result to \cite[Th. 6.1]{BBH} of the functoriality properties of the local element defined in \S\ref{S: local elt}. In this way, for any $H\in \Xi(I_w)$ with $L_w':=(L_w)^H$, one has $q^{I_w}_{I_w/H}(\mathfrak{a}_{L_w/B_w})=\mathfrak{a}_{L_w'/B_w}$. Moreover, by \cite[Th. 8.1]{BBH} one has $\mathfrak{a}_{L_w'/B_w}=\mathfrak{c}_{L_w'/B_w}=0$. Hence the result that $\mathfrak{a}_{L_w'/B_w}$ is annihilated by $|I_w|/p$ now follows from Theorem \ref{Prop: odd degree tor group order}(ii) and (iii).
\end{remark}

\section{The results of Vinatier}
If $L/K$ is any extension as in Theorem \ref{Thm: p-group main intro}, the result of Theorem \ref{Thm: p-group main intro} can be combined with the result \cite[Th. 5.2(iv)]{BBH} of Bley, Burns and Hahn to deduce that 
\[ \frac{n(L/K)}{p}\cdot [\mathcal{A}_{L/K}] = 0\]
in ${\rm Cl}(\ZZ[G])$. 

In the special case that $K = \QQ$, Vinatier \cite{V} has given a much better upper bound on the order of $[\mathcal{A}_{L/\QQ}]$, with additional improvements in the case $p=3$ obtained in \cite{V2}. The essential part in proving this bound is that, by working in ${\rm Cl}(\ZZ[G])$, Vinatier is able to prove a `global' analogue of the induction formula in \eqref{p group ind} in which the exponent $|\Gamma/\Delta|$ is replaced by $p$. (We note that the second expression of \eqref{p group ind} is an induction in the sense of \eqref{ind formula}.) 

In the next result, we describe a consequence of Vinatier's approach for the relative element $\Delta(\mathcal{A}_{L/\QQ})$. 

We note that while (as far as we can see) this result does not directly provide evidence in support of Conjecture~\ref{bbh conj}, it does provide further evidence for the general belief that the elements $\Delta(\mathcal{A}_{L/K})$ should be closely controlled by analytic invariants. 

\begin{thm}\label{vinatier lift} 
Fix an odd prime $p$ and a wildly and weakly ramified Galois extension $L$ of $\QQ$ of $p$-power degree and set $G := \Gal(L/\QQ)$. We fix a $p$-adic place $w$ of $L$ and write $e(w)$ for its absolute ramification degree and $H$ for its decomposition subgroup in $G$ and identify the latter subgroup with $\Gal(L_w/\QQ_p)$. Let ${\rm i}^*_w$ denote the induction map $K_0(\ZZ[H],\QQ^c[H]) \to K_0(\ZZ[G],\QQ^c[G])$ and $\delta_w$ for the composite homomorphism $\partial^1_{\ZZ,\QQ^c,H}\circ ({\rm Nrd}_{\QQ^c[H]})^{-1}$. Then the element 
\[ \mathfrak{a}_{L/\QQ} + {\rm i}^*_{w}\bigl(\delta_{w}((\psi_{2,\ast}-1)(\tau'_{L_w/\QQ_p}))\bigr)\]
of $K_0(\ZZ[G],\QQ^c[G])$ is annihilated by $(p-1)e(w)$.   
\end{thm}

\begin{proof} 
We set $\delta := \delta_{\ZZ,\QQ^c,G}$, $\widetilde \tau_w := (\psi_{2,\ast}-1)(\tau'_{L_w/\QQ_p})$ and abbreviate the map $\mathrm{\tilde{i}}_{H}^{G}: \zeta(\QQ^c[H])^\times \to \zeta(\QQ^c[G])^\times$ to $\tilde {\rm i}_w$. Then, since ${\rm i}^*_{w}\circ \delta_w = \delta\circ \tilde {\rm i}_w$ (see (\ref{Diagram: ind})), the stated element in the claim is equal to 
\[ \xi := \mathfrak{a}_{L/\QQ} + \delta(\tilde {\rm i}_w(\widetilde \tau_w)).\]

First, we note that the horizontal isomorphism in \cite[Chap. II, \S1, Lem. 1.6]{F83} implies that, in terms of the decomposition \eqref{Eq: element in centre}, for every finite prime $\ell$, $a\in \mathrm{GL}_n(\QQ_\ell^c[G])$ and $\chi$ in $\widehat{G}$, one has that $\mathrm{Nrd}_{\QQ_\ell^c[G]}(a)_{\chi} = \Det_{\chi}(a)$, where $\Det$ denotes the generalized determinant discussed in \cite[Chap. I, \S2]{F83}. In this way $\mathrm{Nrd}_{\QQ_\ell^c[G]}(a)$ can be considered as $\Det(a): \chi \mapsto \Det_{\chi}(a)$, an element of $\Hom(R_{G}, \QQ^{c\times}_\ell)$. Write $\mathcal{M}$ for the maximal order in $\QQ[G]$ that contains $\ZZ[G]$, we then deduce from \cite[Prop. 2.2, pp. 23]{F83} that,
\[ \prod_\ell {\rm Nrd}_{\QQ_\ell[G]}(\mathcal{M}_\ell^{\times}) ={\rm Hom}(R_{G}, U_f(\QQ^c))^{\Omega_{\mathbb{Q}}}, \]
where the product runs over all finite primes $\ell$ and $U_f(\QQ^c)$ denotes the group of unit ideles (i.e. the ideles $u$ whose $\ell$-components are units for all finite primes $\ell$). 
	
Then, the fact that the values of the unramified characteristic are roots of unity in $\QQ^c$ (cf. \cite[Th. 29(i)]{F83}) and the definition of the induction map \eqref{Eq: center ind def} combine to imply that $\tilde {\rm i}_w((\psi_{2,\ast}-1)(y_{L_w/\QQ_p}^{-1}))$ belongs to $\prod_\ell{\rm Nrd}_{\QQ_\ell[G]}(\mathcal{M}_\ell^{\times})$. In addition, the result \cite[Prop. 2.5]{V} of Vinatier implies that the $p$-th power of the element $\tilde {\rm i}_w((\psi_{2,*}-1)(\tau_{L_w/\QQ_p}))$ also belongs to  $\prod_\ell {\rm Nrd}_{\QQ_\ell[G]}(\mathcal{M}_\ell^{\times})$.
In view of the isomorphisms (\ref{direct sum decomp DT}) and (\ref{tor}), we can therefore deduce that the element
\[ p\cdot\delta(\tilde {\rm i}_w(\widetilde \tau_w)) = \delta\bigl(\tilde {\rm i}_w((\psi_{2,*}-1)(\tau_{L_w/\QQ_p}))^p\bigr) + \delta\bigl(\tilde {\rm i}_w((\psi_{2,\ast}-1)(y_{L_w/\QQ_p}^{-1}))^p\bigr)\]
belongs to $ {\rm DT}(\ZZ[G]) = {\rm DT}(\ZZ_p[G])$ (since $\mathcal{M}_{\ell}=\ZZ_\ell[G]$ for all $\ell \neq p$, see \eqref{tor} and \eqref{useful decomp}).
	
Since $\mathfrak{a}_{L/\QQ}\in {\rm DT}(\ZZ_p[G])$ (by Proposition~\ref{Th: BBH thm}(i)), it follows that the element $p\cdot \xi$ also belongs to the subgroup ${\rm DT}(\ZZ_p[G])$ of $K_0(\Z_p[G], \QQ_p[G])$. For each prime $\ell$, we identify $K_0(\ZZ_\ell[G], \QQ_\ell[G])$ with a subgroup of $K_0(\ZZ[G], \QQ[G])$ by means of decomposition \eqref{K_0 proj K_0_p} and recall the homomorphism $j^c_{p, *}$ from \eqref{jlc def}. In this way the image $j_{p, \ast}^c(p\cdot \xi)$ is equal to $p\cdot \xi$ in $K_0(\Z_p[G], \QQ_p[G])$.
	
Set $\delta_{w,p} :=\delta^1_{\ZZ_p,\QQ_p^c,H}$ and let ${\rm i}^*_{w,p}$ denote the induction homomorphism $ K_0(\ZZ_p[H],\QQ_p^c[H]) \to K_0(\ZZ_p[G],\QQ_p^c[G])$. We note that the decomposition result in \eqref{Eq: frak a decomp} and Remark \ref{tame local remark} combine to imply that $\mathfrak{a}_{L/\QQ}={\rm i}^*_{w,p} (\mathfrak{a}_{L_w/\QQ_p})$. Then, 
\begin{align}\label{reduction} 
p\cdot \xi
=&\, p\cdot \bigl(j^c_{p, \ast} \bigl(\mathfrak{a}_{L/\QQ} + \delta(\tilde {\rm i}_w(\widetilde \tau_w)) ) \bigr) \nonumber \\
=&\, p\cdot {\rm i}^*_{w,p}\bigl(\mathfrak{a}_{L_w/\QQ_p} + \delta_{w,p} (j_p^c(\widetilde \tau_w))\bigr) \nonumber \\
=&\, p\cdot {\rm i}^*_{w,p}\bigl( \Delta(\mathcal{A}_{L_w/\QQ_p})- \delta_{w,p}\bigl( j_p^c(T^{(2)}_{L_w/\QQ_p})\bigr) - U_{L_w/\QQ_p} + \delta_{w,p} (j_p^c(\widetilde \tau_w))\bigr)\nonumber\\
=&\, p\cdot {\rm i}^*_{w,p}\bigl( \Delta(\mathcal{A}_{L_w/\QQ_p}) - U_{L_w/\QQ_p}\bigr).
\end{align}
Here, the second equality follows from \eqref{Diagram: ind} that $\delta \circ \tilde {\rm i}_w= {\rm i}^*_{w} \circ \delta_w$, the fact that $j_{p, *}^c \circ {\rm i}^*_{w} ={\rm i}^*_{w,p} \circ j_{p, *}^c$ (directly from the definitions \eqref{Eq: K_0 ind def} and \eqref{jlc def}) and \cite[pp. 580]{BB} that $ j_{p, \ast}^c \circ \delta_w = \delta_{w, p}  \circ j_{p}^c$.
The third equality follows directly from the explicit definition of $\mathfrak{a}_{L_w/\QQ_p}$ (see Definition \ref{local a def}), and the final one is true since $T^{(2)}_{L_w/\QQ_p} = \widetilde \tau_w$ as $\tau^H_{\QQ_p} = 1$. 
	
Now, since $p\cdot \xi$ belongs to ${\rm DT}(\ZZ_p[G])$, Theorem \ref{Prop: odd degree tor group order}(i) implies that  
\[ (p-1)|G|\cdot \xi = (1-1/p)|G|\cdot (p\cdot \xi) = 0.\]
	
As $e(w)$ is a non-trivial power of $p$, in order to prove that $(p-1)e(w)\cdot\xi = 0$, by Taylor's Fixed Point Theorem, it is therefore enough to show that the image in $K_0(\mathcal{O}_p^t[G],\QQ_p^c[G])$ of the element $e(w)\cdot \xi = \bigl(e(w)/p\bigr) \cdot (p\cdot \xi)$ is torsion and of order prime to $p$. 

Since the image of $U_{L_w/\QQ_p}$ in $K_0(\mathcal{O}_p^t[H],\QQ_p^c[H])$ vanishes (see \cite[Lem. 4.4]{B_phd}), the expression (\ref{reduction}) reduces us to showing that the image in $K_0(\mathcal{O}_p^t[H],\QQ_p^c[H])$ of $e(w)\cdot \Delta(\mathcal{A}_{L_w/\QQ_p})$ is torsion and of order prime to $p$. 
	
To prove this, we fix an element $a_p$ of $L_w$ such that $\mathcal{A}_{L_w/\QQ_p}= \ZZ_p[H]\cdot a_p$. Then, since the element $\delta_{F}$ (defined in \eqref{delta def}) is equal to $1$ when $F = \QQ_p$, the formula \cite[Lem. 4.16]{B_phd} of Breuning implies that, in $K_{0}(\mathbb{Z}_{p}[H], \mathbb{Q}^c_{p}[H])$, one has 
\[ \Delta(\mathcal{A}_{L_w/\QQ_p}) = \delta_{w,p} \left(\sum_{\chi\in \widehat{H}_p} (a_p|\chi)\cdot e_{\chi} \right). \]
	
We next note that, by the result \cite[Prop. 2.7]{V} of Vinatier, there exists a finite unramified extension $E$ of $\QQ_p$ such that if $q$ is the order of the residue field of $E$, for all $\chi\in \widehat{H}_p$, one has
\[ (a_p|\chi)^{(q-1)e(w)} \in {\rm Det}_\chi (\mathcal{O}_E[H]^\times). \]
It follows that, in terms of \cite[Chap. II, \S2, Lem. 2.1 and \S1, Lem. 1.6]{F83}, one has
\[ \bigl(\sum_{\chi\in \widehat{H}_p} (a_p|\chi)\cdot e_{\chi}\bigr)^{(q-1)e(w)} \in {\rm Nrd}_{\QQ_p^c[H]}\bigl( K_1(\mathcal{O}_E[H]) \bigr) .\] 
In particular, since the image of this element under $\delta_{w,p}$ represents $(q-1)e(w)\cdot \Delta(\mathcal{A}_{L_w/\QQ_p})$, and the composite homomorphism 
\[ K_1(\mathcal{O}_E[H]) \to K_1(\QQ_p^c[H]) \to K_0(\mathcal{O}_p^t[H],\QQ_p^c[H])\]
is zero (as $E/\QQ_p$ is unramified), we deduce that the image in $K_0(\mathcal{O}_p^t[H],\QQ_p^c[H])$ of $e(w)\cdot \Delta(\mathcal{A}_{L_w/\QQ_p})$ is torsion, and of order dividing $q-1$. 
	
Since $q-1$ is prime to $p$, this implies the result we desire and hence, completes the proof.\end{proof} 

We shall finish this article by making some remarks concerning the class of $[\mathcal{A}_{L/\QQ}]$ in the class group.

\begin{remark}\label{fin remark}
If we project the result of Theorem~\ref{vinatier lift} into the class group, the bound we obtain on the order of $[\mathcal{A}_{L/\QQ}]$ is weaker than the one proved by Vinatier in \cite[Th. 1]{V}. This is due to the fact that, in the class group, one knows that the class of $[\mathcal{A}_{L/\QQ}]$ belongs to the kernel subgroup ${\rm D}(\ZZ[G])$, which is again a $p$-group if $G$ has $p$-powered order (cf. \cite[pp. 404]{V}). However, this result is not shared by the torsion subgroup ${\rm DT}(\ZZ[G])$, and therefore we are unable to obtain a sharper bound in $K_0(\ZZ[G],\QQ^c[G])$.
\end{remark}

\begin{coro}
Under the hypothesis of Theorem~\ref{vinatier lift}, one has that $e(w) \cdot [\mathcal{A}_{L/\QQ}]=0$ in the class group. Moreover, if $L_w/\QQ_p$ is totally ramified, then the order of $[\mathcal{A}_{L/\QQ}]$ is bounded by $e(w)/p$.
\end{coro}
\begin{proof}
The first claim will be proved by combining the results of Theorem \ref{Thm: p-group main intro} and \ref{vinatier lift}. We first note that ${\rm i}^*_{w}\circ \delta_w = \delta\circ \tilde {\rm i}_w =\partial^1_{\ZZ,\QQ^c,G}\circ ({\rm Nrd}_{\QQ^c[G]})^{-1} \circ \tilde {\rm i}_w$. Then the image in $\Cl(\Z[G])$ of the correction term ${\rm i}^*_{w}\bigl(\delta_{w}((\psi_{2,\ast}-1)(\tau'_{L_w/\QQ_p}))\bigr)$ in Theorem~\ref{vinatier lift} vanishes since $\partial^0_{\ZZ,\QQ^c,G} \circ \partial^1_{\ZZ,\QQ^c,G}$ is a zero homomorphism (by the exactness of \eqref{Diagram: Kseq}). We also recall from \cite[Th. 5.2(iv)]{BBH} that $\partial_{\ZZ, \QQ^c, G}^{0} (\mathfrak{a}_{L/\QQ})= [\mathcal{A}_{L/\QQ}]$ and therefore deduce that, under the surjective map $K_0(\ZZ[G],\QQ^c[G])\to \Cl(\ZZ[G])$, the order of $[\mathcal{A}_{L/\QQ}]$ is bounded by 
\[
\gcd(n(L/K)/p, (p-1)e(w)) = e(w).
\]
In this way, we obtain the result of Vinatier. 

Now suppose that $|H|= e(w)$, the argument above Remark \ref{remark above} implies that $e(w)/p \cdot \mathfrak{a}_{L/\QQ}=0$ in $K_0(\ZZ[G],\QQ^c[G])$ and so claim (ii) is true.
\end{proof}



\end{document}